\documentclass[11pt]{article}

\usepackage[margin=1.1in]{geometry}
\usepackage{amsmath,amssymb,amsthm,mathtools}
\usepackage{enumitem}
\usepackage{xcolor}
\usepackage{aliascnt}
\usepackage[colorlinks=true,linkcolor=blue,citecolor=blue,urlcolor=blue]{hyperref}
\numberwithin{equation}{section}

\newtheorem{theorem}{Theorem}[section]
\newaliascnt{proposition}{theorem}
\newtheorem{proposition}[proposition]{Proposition}
\aliascntresetthe{proposition}
\newaliascnt{lemma}{theorem}
\newtheorem{lemma}[lemma]{Lemma}
\aliascntresetthe{lemma}
\newaliascnt{corollary}{theorem}

\aliascntresetthe{corollary}
\newaliascnt{remark}{theorem}
\newtheorem{remark}[remark]{Remark}
\aliascntresetthe{remark}

\usepackage[nameinlink,capitalise]{cleveref}
\crefname{theorem}{Theorem}{Theorems}
\Crefname{theorem}{Theorem}{Theorems}
\crefname{proposition}{Proposition}{Propositions}
\Crefname{proposition}{Proposition}{Propositions}
\crefname{lemma}{Lemma}{Lemmas}
\Crefname{lemma}{Lemma}{Lemmas}
\crefname{corollary}{Corollary}{Corollaries}
\Crefname{corollary}{Corollary}{Corollaries}
\crefname{remark}{Remark}{Remarks}
\Crefname{remark}{Remark}{Remarks}

\newcommand{\dd}{\,d}
\newcommand{\one}{\mathbf{1}}
\newcommand{\supp}{\operatorname{supp}}

\title{{\Large\bf Physical-Space Scarring in Generic Bunimovich Stadia}}
\author{Heng~Li\thanks{School of Mathematics, Shandong University, Jinan, China, and Extremal Combinatorics and Probability Group, Institute for Basic Science, Daejeon, South Korea. Email: \texttt{heng.li@sdu.edu.cn}.}
\and
Xizhi~Liu\thanks{School of Mathematical Sciences, University of Science and Technology of China, Hefei, China. Email:~\texttt{liuxizhi@ustc.edu.cn}.}}
\date{\today}

\begin{document}
\maketitle

\begin{abstract}
For the family of Dirichlet stadia $S_t$ whose rectangular part has height $\pi$ and half-length $\pi t/2$, $t \in [1,2]$, we show that for Lebesgue almost every $t$ there exist real eigenfunctions $u_j$ and a smooth mean-zero physical observable $Q$ for which $\langle Q u_j,u_j\rangle$ has a non-zero subsequential limit.  Consequently, along the same subsequence, the eigenfunction mass fails to equidistribute on a fixed region whose relative boundary in the interior of the stadium is smooth.  This proves a physical-space strengthening of Hassell's non-QUE theorem for generic stadia, and thus gives an affirmative answer to Tao's question in Hassell's generic setting.

The proof uses the classification of generic stadia in Hassell's argument.  In each of the resulting cases, we construct an appropriate physical observable $Q$ that converts Hassell's phase-space obstruction to QUE into physical-space non-equidistribution.
\end{abstract}

\section{Introduction}

For every real number $t \in [1,2]$, let $S_t$ be the Bunimovich stadium obtained by adjoining two semicircular caps of radius $\pi/2$ to the rectangle $R_t\coloneqq I_t\times(-\pi/2,\pi/2)$, where $I_t\coloneqq(-\pi t/2,\pi t/2)$. 
Let
\[
H_t\coloneqq-\Delta^D_{S_t}
\]
be the Dirichlet Laplacian.  We use the sign convention $H_t u=\lambda u$, $\lambda>0$.  By smooth functions on $\overline S_t$ we mean restrictions to $\overline S_t$ of smooth functions defined in an open neighbourhood of $\overline S_t$.
By a physical observable we mean a multiplication observable on configuration space: a function $Q=Q(z)$ on $\overline S_t$, depending only on the position $z$ and not on a cotangent direction.  Thus a continuous, respectively smooth, physical observable is simply an element of $C(\overline S_t)$, respectively $C^\infty(\overline S_t)$.

A normalized eigenfunction sequence $u_j$ is \emph{physically equidistributed} if, for every continuous physical observable $Q$ on $\overline S_t$,
\[
 \int_{S_t} Q(z)|u_j(z)|^2\dd z
 \to
 \frac1{|S_t|}\int_{S_t} Q(z)\dd z.
\]
Here the convergence is as $j\to\infty$; equivalently, the probability measures $|u_j(z)|^2\dd z$ converge weakly to the normalized Lebesgue measure $|S_t|^{-1}\dd z$ when tested against continuous physical observables.
It is \emph{physically scarred} if this convergence fails for at least one fixed physical observable.  For regions whose boundary has area zero, the corresponding mass formulation is equivalent; in the theorem below the witnessing region will have smooth relative boundary in the interior of the stadium.  In that region formulation, ``smooth boundary'' means smooth relative boundary in the interior of the billiard domain; possible intersections with the billiard boundary are irrelevant for area measure.

The stadium is a central example in quantum chaos: its billiard flow is ergodic, as shown by Bunimovich \cite{Bunimovich1979}, but it has singular invariant measures carried by the vertical bouncing-ball trajectories.  The quantum ergodicity theorem of Shnirelman, Colin de Verdi{\`e}re, and Zelditch \cite{Shnirelman1974,ColinDeVerdiere1985,Zelditch1987}, together with its billiard versions due to G{\'e}rard and Leichtnam \cite{GerardLeichtnam1993} and to Zelditch and Zworski \cite{ZelditchZworski1996}, gives equidistribution for a density-one subsequence of eigenfunctions.  Boundary versions and related billiard refinements were developed by Hassell and Zelditch \cite{HassellZelditch2004}.  The stronger QUE problem asks whether exceptional subsequences can be excluded; see, for example, Rudnick and Sarnak's formulation in negative curvature \cite{RudnickSarnak1994}.  In the stadium, the bouncing-ball quasimodes behind the physical and numerical scarring picture go back to Heller's work \cite{Heller1984} and to the localization mechanism of O'Connor and Heller \cite{OConnorHeller1988}; see also Zelditch's note on the associated quasimode mechanism \cite{Zelditch2004}.  Rigorous results on how stadium eigenfunctions and quasimodes must spread into the wings were obtained by Burq and Zworski \cite{BurqZworski2005} and by Burq, Hassell, and Wunsch \cite{BurqHassellWunsch2007}, while numerical studies of rates of quantum ergodicity in Euclidean billiards, including the stadium, appear in B{\"a}cker, Schubert, and Stifter \cite{BackerSchubertStifter1998}.

Tao formulated the physical-space scarring problem for the Bunimovich stadium as an open question \cite{Tao2007}.  Hassell subsequently proved that, for Lebesgue almost every member of the standard one-parameter family of stadia, the Dirichlet eigenfunctions are not quantum uniquely ergodic \cite{Hassell2010}.  His conclusion is microlocal: there is a high-energy subsequence whose semiclassical measures in phase space are not the Liouville measure.  The obstruction comes from the bouncing-ball dynamics in the rectangular part of the stadium and is extracted from bounded spectral windows around the corresponding bouncing-ball quasimodes.  Tao noted that upgrading this phase-space conclusion to physical-space non-equidistribution remained a natural direction \cite{Tao2008}.  The point of the present note is that Hassell's bounded-window output already contains enough information to make this physical-space extraction.

\begin{theorem}\label{thm:main}
For Lebesgue almost every $t\in[1,2]$ there exist normalized real Dirichlet eigenfunctions
\[
H_t u_j=\lambda_j u_j,\qquad \|u_j\|_{L^2(S_t)}=1,\qquad\text{and}\qquad \lambda_j\to\infty,
\]
and a real-valued $Q\in C^\infty(\overline S_t)$ such that $\int_{S_t} Q(z)\dd z=0$ and, after passing to a subsequence,
\[
\int_{S_t} Q(z)|u_j(z)|^2\dd z\to \gamma\neq 0.
\]
Consequently, along the same subsequence, there is a physical region $A\subseteq S_t$, whose relative boundary in the interior of the stadium may be chosen smooth, such that
\[
 \int_A |u_j(z)|^2\dd z \not\to \frac{|A|}{|S_t|}.
\]
\end{theorem}


The proof keeps Hassell's argument as a black box and uses only two consequences of his non-QUE theorem for generic stadia \cite{Hassell2010}.  For the stadium family $S_t$, $t\in[1,2]$, Hassell's proof classifies the parameter interval $[1,2]$ into two alternatives, denoted here by $Z_1$ and $Z_2$: the first produces a semiclassical-measure obstruction, while the second produces a bounded-window obstruction near bouncing-ball quasimodes.  In the $Z_1$ alternative, the limiting mass avoids the caps in the microlocal sense needed here, so a physical observable supported in the cap interiors detects non-equidistribution.  In the complementary $Z_2$ alternative, Hassell's bounded-window argument gives eigenfunctions with a fixed overlap with a bouncing-ball quasimode; the local Fourier analysis in the central rectangle then converts this overlap into a smooth mean-zero physical observable with non-zero limiting expectation.  Finally, a layer-cake argument turns such an observable into a physical region whose mass is not equidistributed.

The organization is as follows.  \Cref{prop:hassell-input} records the two Hassell consequences in the form used later.  The one-dimensional estimates needed for the local extraction are proved next, followed by the fixed-overlap extraction in \Cref{prop:fixed-overlap}.  We then pass from smooth observables to regions and assemble the proof of the theorem.  The $Z_1$ semiclassical-measure alternative uses the later clarification of Mangoubi and Weller Weiser \cite{MangoubiWeller2024} concerning singular boundary points in Hassell's stadium argument; the bounded-window extraction is purely local to the central rectangle and does not involve that boundary issue.

\section{Hassell's classification}

This section isolates the part of Hassell's argument that will be used in the
rest of the paper.  We do not revisit the proof of the generic non-QUE theorem;
instead, we record its two outputs in a form adapted to the present physical-space
extraction.  We first fix a normalized family of bouncing-ball quasimodes,
uniformly for the stadia $S_t$, and then state the resulting dichotomy for the
parameter interval.

Set
\[
\varphi_n(y)\coloneqq\sqrt{\frac2\pi}\,\sin\bigl(n(y+\pi/2)\bigr)
\quad\text{for}\quad y\in(-\pi/2,\pi/2).
\]
Then $-\partial_y^2\varphi_n=n^2\varphi_n$ and $\varphi_n(\pm\pi/2)=0$.  Fix once and for all a real-valued cutoff
\[
\chi\in C_c^\infty((-\pi/4,\pi/4)),\qquad
\|\chi\|_{L^2(\mathbb R)}=1,
\qquad\text{and}\qquad
\int_{\mathbb R}\chi(x)\dd x=0.
\]
For every $t\in[1,2]$, define
\begin{equation}\label{eq:bouncing-ball-quasimode}
v_n(x,y)\coloneqq\chi(x)\varphi_n(y)
\end{equation}
in the central rectangle and extend it by zero to $S_t$.  The support of $\chi$ is a positive distance away from the two interfaces with the caps, and $\varphi_n$ satisfies the Dirichlet condition on the horizontal sides.  Hence this extension belongs to the domain of the Dirichlet Laplacian and
\[
\|v_n\|_{L^2(S_t)}=1,
\qquad\text{and}\qquad
(H_t-n^2)v_n=(-\chi''(x))\varphi_n(y).
\]
Put $K_\chi\coloneqq\|\chi''\|_{L^2(\mathbb R)}$. 
Then
\begin{equation}\label{eq:quasimode-error}
\|(H_t-n^2)v_n\|_{L^2(S_t)}=K_\chi
\end{equation}
uniformly in $n$ and $t$.  Notice that $K_\chi>0$: otherwise $\chi$ would be affine on its support and hence, by compact support, identically zero, contradicting $\|\chi\|_2=1$.

Here and below, $A^\circ$ denotes the interior of a set $A$, and
$C_c^\infty(U)$ denotes the space of smooth functions on $U$ with compact
support contained in $U$.  We write
\[
        S^*S_t^\circ\coloneqq
        \{(z,\zeta):z\in S_t^\circ,\ \zeta\in T_z^*S_t^\circ,\ |\zeta|=1\}
\]
for the unit cotangent bundle over the interior of the stadium.  The
superscript $*$ denotes the cotangent, or dual, object: for
$z\in S_t^\circ$, the notation $T_z^*S_t^\circ$ denotes the cotangent
space of $S_t^\circ$ at $z$, while $S^*S_t^\circ$ denotes the corresponding
unit cotangent bundle.  Since $S_t^\circ\subseteq\mathbb R^2$ is an open
planar domain, $T_z^*S_t^\circ$ is naturally identified with
$(\mathbb R^2)^*$, and hence $S^*S_t^\circ$ may be identified with
$S_t^\circ\times S^1$.

Hassell's proof of Theorem 4 in \cite{Hassell2010} splits the parameter set into two parts, denoted here by $Z_1$ and $Z_2$; see his equation (9).  The following proposition is an extracted corollary of that proof, specialized to the simple stadium family above and to the fixed cutoff $\chi$.

\begin{proposition}\label{prop:hassell-input}
Let $Z_1,Z_2$ be the partition of $[1,2]$ arising in Hassell's proof for the Dirichlet stadium family, where $Z_1$ is the set defined by Hassell's condition $\liminf_{j\to\infty} f_j(t)=0$, and $Z_2\coloneqq[1,2]\setminus Z_1$.  With $v_n$ defined by \eqref{eq:bouncing-ball-quasimode}, the following assertions hold.
\begin{enumerate}[label=\textnormal{(\alph*)}]
\item For every $t\in Z_1$, write $\mathcal C_t^\circ\coloneqq S_t^\circ\setminus \overline{R_t}$ for the union of the interiors of the two caps, and let
\[
        \operatorname{pr}:S^*S_t^\circ\to S_t^\circ,\qquad
        \operatorname{pr}(z,\zeta)=z
\]
be the base projection.  There exist normalized Dirichlet eigenfunctions $w_j$ and eigenvalues $\lambda_j$ such that
\[
        H_t w_j=\lambda_j w_j,\qquad
        \|w_j\|_{L^2(S_t)}=1,\qquad
        \lambda_j\to\infty,
\]
and an associated interior semiclassical measure $\mu$ satisfying
\begin{equation}\label{eq:z1-cap-avoidance}
        \mu\bigl(\operatorname{pr}^{-1}(\mathcal C_t^\circ)\bigr)=0.
\end{equation}
In particular, for every $Q_0\in C_c^\infty\bigl(S_t^\circ\setminus \overline{R_t}\bigr)$, one has, after passing to the subsequence defining $\mu$,
\[
        \langle Q_0 w_j,w_j\rangle_{L^2(S_t)}\to 0.
\]
In Hassell's formulation, modulo the standard gliding and singular boundary sets, the corresponding quantum limit is supported on the bouncing-ball covectors in the rectangular part.

\item For Lebesgue almost every $t\in Z_2$ there exist integers $n_j\to\infty$, normalized real Dirichlet eigenfunctions $u_j$, and constants $C_t,c_t>0$ such that
\[
        H_t u_j=\lambda_j u_j,\qquad
        |\lambda_j-n_j^2|\leq C_t,\qquad\text{and}\qquad
        |\langle u_j,v_{n_j}\rangle_{L^2(S_t)}|\geq c_t
        \quad\text{for every}\quad j.
\]
\end{enumerate}
\end{proposition}

\begin{remark}\label{rem:z1-cap-avoidance}
The condition \eqref{eq:z1-cap-avoidance} says that the limiting microlocal mass of this subsequence does not lie over the interiors of the two semicircular caps.  The map $\operatorname{pr}$ forgets the covector direction and remembers only the base point $z$, so $\operatorname{pr}^{-1}(\mathcal C_t^\circ)$ is the part of phase space consisting of all unit covectors based at points inside the two caps.  Hence any physical observable supported compactly in the cap interiors has limiting expectation zero along the subsequence defining $\mu$.
\end{remark}

\begin{proof}[Proof of \Cref{prop:hassell-input}]
We first prove the $Z_1$ statement in the precise form needed later.
Hassell's $Z_1$ argument gives a boundary semiclassical measure that
vanishes on the curved sides of the stadium.  The relation between boundary
measures and interior quantum limits then gives an interior quantum limit
supported on rays that do not meet the curved sides.  For the Bunimovich
stadium, the nonsingular interior rays with this property are the
bouncing-ball rays in the rectangular part.  At the four curvature-jump
points of the stadium boundary, we use the clarification of Mangoubi and
Weller Weiser \cite{MangoubiWeller2024}; it shows that Hassell's use of
the non-gliding-point argument still gives the conclusion needed here: the
interior measure obtained in the $Z_1$ alternative gives no mass to compact
subsets of the phase-space region lying over the cap interiors.  By inner
regularity, this gives
\[
        \mu\bigl(\operatorname{pr}^{-1}(\mathcal C_t^\circ)\bigr)=0.
\]
Hence, for any
$Q_0\in C_c^\infty(S_t^\circ\setminus \overline{R_t})$, multiplication by
$Q_0$ is an interior zeroth-order observable with principal symbol
$Q_0\circ\operatorname{pr}$.  Since $Q_0\circ\operatorname{pr}$ vanishes on $\operatorname{supp}\mu$,
\[
        \langle Q_0 w_j,w_j\rangle_{L^2(S_t)}
        \to
        \int_{S^*S_t^\circ} Q_0(\operatorname{pr}(z,\zeta))\,\dd\mu(z,\zeta)
        =0.
\]
This proves part \textnormal{(a)}.

We turn to part \textnormal{(b)}.  Recall that $\|(H_t-n^2)v_n\|_{L^2(S_t)}=K_\chi$ uniformly in $n$ and $t$.  Set $a_\chi\coloneqq2K_\chi>0$.  We shall apply Hassell's bounded-window argument with this fixed half-width $a\coloneqq a_\chi$.

We first record explicitly the consequence of Hassell's parameter estimate
that we need.  For fixed $a>0$ and $\varepsilon>0$, there is a measurable
set $F_{\varepsilon,a}\subseteq Z_2$ such that
$|Z_2\setminus F_{\varepsilon,a}|\leq 4\varepsilon$, and there is a finite integer $M_{\varepsilon,a}$ with the following
property: for every $t\in F_{\varepsilon,a}$ there are infinitely many
integers $n$ such that the closed spectral window $[n^2-a,n^2+a]$
contains at most $M_{\varepsilon,a}$ eigenvalues of $H_t$, counted with
multiplicity.

Indeed, this is precisely the content of Hassell's equations (10)--(15),
but let us spell out the extraction.  Since $t\in Z_2$ means
$\liminf_j f_j(t)>0$, Hassell first chooses, for each $\varepsilon>0$, a
set $G_\varepsilon\subseteq Z_2$ with
$|G_\varepsilon|\geq |Z_2|-2\varepsilon$, together with constants $c_\varepsilon>0$ and $N_\varepsilon$, such that
\[
        t\in G_\varepsilon,\quad j\geq N_\varepsilon
        \quad\Longrightarrow\quad
        f_j(t)\geq \frac{c_\varepsilon}{2}.
\]
Using the eigenvalue variation formula and Weyl bounds, Hassell obtains a
constant $C_\varepsilon>0$, independent of $a$ and $n$, such that, for all
sufficiently large $n$,
\[
        \int_{G_\varepsilon}
        \#\{\lambda\in \operatorname{Spec}(H_t):|\lambda-n^2|\leq a\}\,\dd t
        \leq
        C_\varepsilon a.
\]
Define $M_{\varepsilon,a} \coloneqq \left\lceil {C_\varepsilon a}/{\varepsilon} \right\rceil$. 
By Chebyshev's inequality, for all sufficiently large $n$ the set
\[
        A_n
        \coloneqq
        \left\{
        t\in G_\varepsilon:
        \#\{\lambda\in \operatorname{Spec}(H_t):|\lambda-n^2|\leq a\}
        \leq M_{\varepsilon,a}
        \right\}
\]
satisfies
\[
        |A_n|\geq |G_\varepsilon|-\varepsilon
        \geq |Z_2|-3\varepsilon.
\]
It remains to pass from large measure for each $A_n$ to membership in
infinitely many $A_n$'s.  For all sufficiently large integers $k$, set
\[
        D_k
        \coloneqq
        \left\{
        t\in Z_2:
        t\in A_n\ \text{for at least $k$ distinct integers } n
        \text{ satisfying } k\leq n<5k
        \right\}.
\]
Then
\[
        \sum_{n=k}^{5k-1}|A_n|
        \geq
        4k(|Z_2|-3\varepsilon).
\]
On the other hand, by the definition of $D_k$, a point of $D_k$ can be
counted at most $4k$ times in this sum, while a point of
$Z_2\setminus D_k$ can be counted at most $k$ times.  Hence
\[
        \sum_{n=k}^{5k-1}|A_n|
        \leq
        4k|D_k|+k(|Z_2|-|D_k|).
\]
Combining the last two estimates gives $|D_k|\geq |Z_2|-4\varepsilon$.
Since $D_k\subseteq \bigcup_{n\geq k}A_n$, we get, for all sufficiently large $k$,
$\left|\bigcup_{n\geq k}A_n\right|\geq |Z_2|-4\varepsilon$.
By continuity of measure from above,
\[
        \left|
        \limsup_{n\to\infty} A_n
        \right|
        =
        \left|
        \bigcap_{k=1}^\infty \bigcup_{n\geq k}A_n
        \right|
        \geq |Z_2|-4\varepsilon.
\]
Taking
\[
        F_{\varepsilon,a}\coloneqq\limsup_{n\to\infty}A_n
\]
proves the bounded-window consequence stated above.

Now choose a sequence $\varepsilon_m\downarrow0$, for instance $\varepsilon_m\coloneqq2^{-m}$ and define
\[
        Z_2^{\mathrm{good}}
        \coloneqq
        \bigcup_{m=1}^\infty F_{\varepsilon_m,a_\chi}.
\]
Since $|Z_2\setminus F_{\varepsilon_m,a_\chi}|\leq 4\varepsilon_m$ for every $m$, we have
\[
        |Z_2\setminus Z_2^{\mathrm{good}}|
        =
        \left|
        \bigcap_{m=1}^\infty
        (Z_2\setminus F_{\varepsilon_m,a_\chi})
        \right|
        \leq
        \inf_m 4\varepsilon_m
        =0.
\]
Thus $Z_2^{\mathrm{good}}$ has full measure in $Z_2$.

Fix $t\in Z_2^{\mathrm{good}}$.  Then $t\in F_{\varepsilon_m,a_\chi}$ for
some $m$.  Let $M_t\coloneqq M_{\varepsilon_m,a_\chi}$.
There are infinitely many integers $n$ such that
\[
        \#\{\lambda\in\operatorname{Spec}(H_t):|\lambda-n^2|\leq a_\chi\}
        \leq M_t.
\]
Choose such an $n$.  Let $P_n\coloneqq\one_{(n^2-a_\chi,n^2+a_\chi)}(H_t)$ be the spectral projection onto the open window.  Its rank is at most
$M_t$, since the open window is contained in the closed window counted
above.  On the complementary spectral subspace,
$|\lambda-n^2|\geq a_\chi$.
Therefore, by the spectral theorem and the quasimode estimate,
\[
        a_\chi\|(I-P_n)v_n\|_{L^2(S_t)}
        \leq
        \|(H_t-n^2)(I-P_n)v_n\|_{L^2(S_t)}
        \leq
        \|(H_t-n^2)v_n\|_{L^2(S_t)}
        =
        K_\chi.
\]
Since $a_\chi=2K_\chi$, this gives
$\|(I-P_n)v_n\|_{L^2(S_t)}\leq \frac12$, and hence
\[
        \|P_n v_n\|_{L^2(S_t)}^2
        =
        1-\|(I-P_n)v_n\|_{L^2(S_t)}^2
        \geq
        \frac34.
\]
Let $\{u_{n,k}\}_{k=1}^{r_n}$ be an orthonormal real eigenbasis of $\operatorname{ran}P_n$.  Such a real
basis may be chosen because $H_t$ has real coefficients and the Dirichlet
condition is real.  Since $r_n\leq M_t$,
\[
        \sum_{k=1}^{r_n}
        |\langle u_{n,k},v_n\rangle_{L^2(S_t)}|^2
        =
        \|P_n v_n\|_{L^2(S_t)}^2
        \geq
        \frac34.
\]
Consequently, for at least one $k$,
\[
        |\langle u_{n,k},v_n\rangle_{L^2(S_t)}|
        \geq
        \sqrt{\frac{3}{4M_t}}.
\]
Selecting one such eigenfunction for each of the infinitely many admissible
integers $n$, and relabelling them as $u_j$ with $n=n_j\to\infty$, we
obtain
\[
        H_t u_j=\lambda_j u_j,\qquad
        |\lambda_j-n_j^2|<a_\chi,\qquad\text{and}\qquad
        |\langle u_j,v_{n_j}\rangle_{L^2(S_t)}|
        \geq
        \sqrt{\frac{3}{4M_t}}.
\]
Thus part \textnormal{(b)} holds with, for example,
\[
        C_t\coloneqq a_\chi=2K_\chi,
        \qquad\text{and}\qquad
        c_t\coloneqq\sqrt{\frac{3}{4M_t}}.
\]
The proof is complete.
\end{proof}


\section{One-dimensional lemmas in the rectangle}

We use three elementary facts about solutions of one-dimensional constant-coefficient equations.  They are stated with constants uniform over the compact parameter ranges needed later.  The notation $J\Subset I$ means that the closure of $J$ is a compact subset of $I$.

\begin{lemma}\label{lem:compact}
Let $T>0$, let $I\coloneqq(-T,T)$, and let $\mu_j$ be a bounded sequence of real numbers.  If
\[
-b_j''=\mu_j b_j,
\qquad\text{and}\qquad
\|b_j\|_{L^2(I)}=1,
\]
then a subsequence converges in $C^1(\overline I)$ and in $L^2(I)$ to a non-zero solution $b_\infty$ of
\[
-b_\infty''=\mu_\infty b_\infty
\]
for some subsequential limit $\mu_\infty$.
\end{lemma}

\begin{proof}
After passing to a subsequence, assume $\mu_j\to\mu_\infty$.  It remains to obtain compactness.  Choose $C$ with $|\mu_j|\leq C$.  For each $|\mu|\leq C$ and each initial vector $\xi=(\xi_0,\xi_1)\in\mathbb R^2$, let $B_{\mu,\xi}$ be the solution of
\[
-B''=\mu B,
\qquad
B(0)=\xi_0,
\qquad\text{and}\qquad B'(0)=\xi_1.
\]
The quadratic form
\[
G_\mu(\xi)\coloneqq\|B_{\mu,\xi}\|_{L^2(I)}^2
\]
is positive definite for every $\mu$, because a non-zero solution of a second-order ODE cannot vanish identically on an interval.  Since $\mu\mapsto G_\mu$ is continuous on the compact interval $[-C,C]$, there is $c_0>0$ such that
\[
G_\mu(\xi)\geq c_0|\xi|^2
\quad\text{for}\quad |\mu|\leq C.
\]
Thus the initial data $(b_j(0),b_j'(0))$ are uniformly bounded.  The ODE then gives uniform bounds for $b_j$, $b_j'$, and $b_j''$ on $\overline I$.  Arzel{\`a}--Ascoli gives a subsequence converging in $C^1(\overline I)$ to a solution of the limiting equation.  The $L^2$ norm of the limit is one, so the limit is non-zero.
\end{proof}

\begin{lemma}\label{lem:oscillation}
Let $T>0$, let $I\coloneqq(-T,T)$, and let $q\in C_c^\infty(I)$ with $\int_I q(x)\dd x=0$.  There is a constant $C_q$, depending on $q$ and $I$, such that, for every real number $k\geq1$ and every real solution of
\[
-a''=k^2 a,
\]
one has
\[
\left|\int_I q(x)|a(x)|^2\dd x\right|
\leq \frac{C_q}{k}\|a\|_{L^2(I)}^2.
\]
\end{lemma}

\begin{proof}
Fix $k\geq1$ and write $a(x)=A\cos(kx)+B\sin(kx)$.
This is the general real solution of the constant-coefficient equation
$a''+k^2a=0$.
Then
\[
|a(x)|^2
=\frac{A^2+B^2}{2}
+\frac{A^2-B^2}{2}\cos(2kx)
+AB\sin(2kx).
\]
The constant term integrates to zero against $q$.  One integration by parts gives
\[
\left|\int_I q(x)e^{2ikx}\dd x\right|
\leq \frac{\|q'\|_{L^1(I)}}{2k}.
\]
Therefore
\[
\left|\int_I q(x)|a(x)|^2\dd x\right|
\leq \frac{C_q'}{k}(A^2+B^2).
\]
It remains to compare $A^2+B^2$ with $\|a\|_2^2$ uniformly for $k\geq 1$.  For $1\leq k\leq K$, set
\[
m_K\coloneqq\min_{\substack{1\leq k\leq K\\ A^2+B^2=1}}
\|A\cos(kx)+B\sin(kx)\|_{L^2(I)}^2.
\]
The parameter set is compact, and no non-zero function $A\cos(kx)+B\sin(kx)$ can vanish identically on $I$, so $m_K>0$.  This gives the comparison for $1\leq k\leq K$.  For $k\geq K$ it follows directly from the identities for the integrals of $\cos^2(kx)$, $\sin^2(kx)$, and $\sin(kx)\cos(kx)$: the corresponding Gram matrix is $T$ times the identity plus an $O(k^{-1})$ matrix.  Taking $K$ large enough gives $A^2+B^2\leq C_I\|a\|_{L^2(I)}^2$ for all $k\geq1$, and the lemma follows.
\end{proof}

\begin{lemma}\label{lem:evanescent}
Let $T>0$, let $I\coloneqq(-T,T)$, and let $J\Subset I$.  There are constants $C,c>0$, depending only on $I$ and $J$, such that, for every real number $\alpha\geq1$ and every real solution of
\[
a''=\alpha^2 a,
\]
one has
\[
\|a\|_{L^2(J)}\leq Ce^{-c\alpha}\|a\|_{L^2(I)}.
\]
\end{lemma}

\begin{proof}
Fix $\alpha\geq1$.  Choose $\delta>0$ such that $J\subseteq[-T+\delta,T-\delta]$.  Every solution has the form $a(x)=P_+ e^{\alpha(x-T)}+P_- e^{-\alpha(x+T)}$.
On $J$ both exponentials are $O(e^{-\alpha\delta})$, and hence
\begin{equation}\label{eq:interior-upper}
\|a\|_{L^2(J)}^2\leq C_J e^{-2\alpha\delta}(|P_+|^2+|P_-|^2).
\end{equation}
On $I$, let
\[
f_1(x)\coloneqq e^{\alpha(x-T)},\qquad\text{and}\qquad f_2(x)\coloneqq e^{-\alpha(x+T)}.
\]
The Gram matrix of $f_1$ and $f_2$ has diagonal entries
\[
\langle f_1,f_1\rangle=\langle f_2,f_2\rangle
=\frac{1-e^{-4\alpha T}}{2\alpha}
\]
and off-diagonal entry $\langle f_1,f_2\rangle=2T e^{-2\alpha T}$.
The diagonal entries are comparable to $\alpha^{-1}$ for $\alpha\geq1$, while the off-diagonal entry is exponentially small.  For large $\alpha$ this gives a lower bound of order $\alpha^{-1}$ for the least eigenvalue, and the remaining bounded range of $\alpha$ is handled by compactness and linear independence of $f_1,f_2$.  Hence, after decreasing the constant if necessary,
\begin{equation}\label{eq:gram-lower}
\|a\|_{L^2(I)}^2\geq c_0\alpha^{-1}(|P_+|^2+|P_-|^2),
\quad\text{for}\quad \alpha\geq1.
\end{equation}
Combining \eqref{eq:interior-upper} and \eqref{eq:gram-lower} gives $\|a\|_{L^2(J)}^2\leq C\alpha e^{-2\alpha\delta}\|a\|_{L^2(I)}^2$.  Choose $0<c<2\delta$.  Since $\alpha e^{-(2\delta-c)\alpha}$ is bounded for $\alpha\geq1$, we obtain $\|a\|_{L^2(J)}^2\leq C e^{-c\alpha}\|a\|_{L^2(I)}^2$.
Replacing $c$ by $c/2$ and adjusting $C$ gives the stated norm estimate.
\end{proof}

\section{Physical extraction from a fixed vertical overlap}

This section is the elementary core of the paper.  It is independent of Hassell's parameter argument.

\begin{proposition}\label{prop:fixed-overlap}
Fix $t\in[1,2]$.  Let $u_j$ be normalized real Dirichlet eigenfunctions on $S_t$ and let $n_j\to\infty$ be integers such that
\[
H_t u_j=\lambda_j u_j,\qquad
|\lambda_j-n_j^2|\leq C,\qquad\text{and}\qquad
|\langle u_j,\chi\varphi_{n_j}\rangle_{L^2(S_t)}|\geq c
\]
for some constants $C,c>0$.
Then, after passing to a subsequence, there exists a real-valued $Q\in C^\infty(\overline S_t)$ with $\int_{S_t} Q(z)\dd z=0$ such that
\[
\langle Q u_j,u_j\rangle_{L^2(S_t)}\to \gamma\neq 0.
\]
\end{proposition}

\begin{proof}
All arguments in this proof are local to the rectangle $R_t$.  For $\ell\geq1$, define the vertical Fourier coefficient
\[
a_{j,\ell}(x)\coloneqq\int_{-\pi/2}^{\pi/2}u_j(x,y)\varphi_\ell(y)\dd y
\quad\text{for}\quad x\in I_t.
\]
Since $u_j$ and $\varphi_\ell$ are real-valued, all $a_{j,\ell}$ are real-valued.  Since $u_j$ solves $(-\partial_x^2-\partial_y^2)u_j=\lambda_j u_j$ in $R_t$ and satisfies the Dirichlet condition on the horizontal sides $y=\pm\pi/2$, testing the equation against functions of the form $\eta(x)\varphi_\ell(y)$, with $\eta\in C_c^\infty(I_t)$, gives, in the distributional sense on $I_t$,
\[
-a_{j,\ell}''(x)+\ell^2 a_{j,\ell}(x)=\lambda_j a_{j,\ell}(x).
\]
Equivalently,
\begin{equation}\label{eq:coeff-ode}
-a_{j,\ell}''=(\lambda_j-\ell^2) a_{j,\ell}
\qquad\text{on }I_t.
\end{equation}
Because the right-hand side belongs to $L^2(I_t)$, the coefficient belongs to $H^2_{\mathrm{loc}}(I_t)$, meaning that it has two weak derivatives in $L^2$ on every compact subinterval of $I_t$, and the constant-coefficient ODE then implies smoothness on $I_t$.  No boundary condition at the endpoints of $I_t$ is used; the caps affect how the coefficients match to the exterior of the central rectangle.

By Parseval in the vertical variable,
\begin{equation}\label{eq:parseval}
\sum_{\ell\geq1}\|a_{j,\ell}\|_{L^2(I_t)}^2
=\|u_j\|_{L^2(R_t)}^2
\leq 1.
\end{equation}
The overlap assumption is exactly
\begin{equation}\label{eq:overlap-coeff}
\left|\int_{I_t}a_{j,n_j}(x)\chi(x)\dd x\right|
\geq c.
\end{equation}
Let $A_j\coloneqq\|a_{j,n_j}\|_{L^2(I_t)}$. 
Since $\|\chi\|_2=1$, \eqref{eq:overlap-coeff} and Cauchy's inequality give $c\leq A_j\leq1$.
Passing to a subsequence, assume
\[
A_j\to A_\infty\in[c,1],
\qquad\text{and}\qquad
\mu_j\coloneqq\lambda_j-n_j^2\to\mu_\infty.
\]
Set $b_j\coloneqq A_j^{-1}a_{j,n_j}$. 
Then
\[
\|b_j\|_{L^2(I_t)}=1,
\qquad\text{and}\qquad
-b_j''=\mu_j b_j.
\]
By \Cref{lem:compact}, after passing to a further subsequence, $b_j\to b_\infty$ in $C^1(\overline I_t)$ and in $L^2(I_t)$, where $\|b_\infty\|_2=1$ and $-b_\infty''=\mu_\infty b_\infty$.  Moreover, the $L^2$ convergence gives
\[
\left|\int_{I_t}b_\infty(x)\chi(x)\dd x\right|
=\lim_{j\to\infty}A_j^{-1}
\left|\int_{I_t}a_{j,n_j}(x)\chi(x)\dd x\right|
\geq c.
\]
Since $\int_{\mathbb R}\chi(x)\dd x=0$, the limit $b_\infty$ is not constant.  The functions are real-valued.  A nonconstant real solution of a constant-coefficient second-order ODE cannot have constant square on the connected interval $I_t$: if $b_\infty^2$ were constant, then $b_\infty$ would have constant absolute value, hence by continuity a constant sign and therefore would itself be constant.  Thus $b_\infty^2$ is not constant.

We now choose a detector in the horizontal variable.  There exists a real-valued $q\in C_c^\infty(I_t)$ such that
\begin{equation}\label{eq:q-choice}
\int_{I_t}q(x)\dd x=0,
\qquad\text{and}\qquad
\int_{I_t}q(x)\,b_\infty(x)^2\dd x\neq0.
\end{equation}
Indeed, choose $\rho\in C_c^\infty(I_t)$ with $\int_{I_t}\rho(x)\dd x=1$.  If the second integral vanished for every real-valued test function $q\in C_c^\infty(I_t)$ with zero integral, then for every $\psi\in C_c^\infty(I_t)$ the test function $\psi-(\int_{I_t}\psi(x)\dd x)\rho$ would give
\[
\int_{I_t}\psi(x)\,b_\infty(x)^2\dd x
=\left(\int_{I_t}\psi(x)\dd x\right)\int_{I_t}\rho(x)\,b_\infty(x)^2\dd x.
\]
Thus $b_\infty^2$ would be constant as a distribution on $I_t$, contradicting the previous paragraph.

Extend $q$ by zero to a function in $C_c^\infty(\mathbb R)$, still denoted by $q$, and define the physical observable by
\[
Q(x,y)\coloneqq q(x)
\]
in a neighbourhood of $\overline S_t$.  Since $\supp q\Subset I_t$, this function vanishes near the two cap interfaces and on the caps, and its restriction gives $Q\in C^\infty(\overline S_t)$.  Also
\begin{equation}\label{eq:mean-zero-Q}
\int_{S_t} Q(z)\dd z
=\pi\int_{I_t}q(x)\dd x=0.
\end{equation}
Because $Q$ is independent of $y$ in $R_t$ and vanishes on $S_t\setminus R_t$, orthogonality of the vertical basis gives
\begin{equation}\label{eq:Q-decomposition}
\langle Q u_j,u_j\rangle_{L^2(S_t)}
=\sum_{\ell=1}^\infty\int_{I_t}q(x)|a_{j,\ell}(x)|^2\dd x.
\end{equation}
The identity follows first for finite vertical sums and then by $L^2$ convergence and boundedness of $q$.

We separate the target mode, the lower vertical modes, and the higher vertical modes.  For the target mode, the $L^2$ convergence $b_j\to b_\infty$ gives
\[
\left|\int_{I_t}q(b_j^2-b_\infty^2)\dd x\right|
\leq \|q\|_\infty \|b_j-b_\infty\|_2
\bigl(\|b_j\|_2+\|b_\infty\|_2\bigr)
\to0.
\]
Therefore
\begin{equation}\label{eq:target-limit}
\int_{I_t}q(x)\,|a_{j,n_j}(x)|^2\dd x
=A_j^2\int_{I_t}q(x)\,b_j(x)^2\dd x
\to
\gamma,
\qquad
\gamma\coloneqq A_\infty^2\int_{I_t}q(x)\,b_\infty(x)^2\dd x.
\end{equation}
By \eqref{eq:q-choice} and $A_\infty\geq c>0$, $\gamma\neq0$.

All constants in the following mode estimates may depend on the fixed data $t$ and $q$, and later on the chosen interval $J$, but not on $j$ or $\ell$.

For $\ell<n_j$, put $k_{j,\ell}\coloneqq\sqrt{\lambda_j-\ell^2}$. 
Since $\lambda_j=n_j^2+O(1)$,
\[
k_{j,\ell}^2\geq n_j^2-C-(n_j-1)^2
=2n_j-1-C
\]
for all $\ell<n_j$ and all large $j$.  Thus $k_{j,\ell}\geq c_0\sqrt{n_j}$, in particular $k_{j,\ell}\geq1$.  By \eqref{eq:coeff-ode} and \Cref{lem:oscillation},
\[
\left|\int_{I_t}q(x)\,|a_{j,\ell}(x)|^2\dd x\right|
\leq C_q n_j^{-1/2}\|a_{j,\ell}\|_{L^2(I_t)}^2.
\]
Summing and using \eqref{eq:parseval},
\begin{equation}\label{eq:lower-modes}
\left|\sum_{\ell<n_j}\int_{I_t}q(x)\,|a_{j,\ell}(x)|^2\dd x\right|
\leq C_q n_j^{-1/2}
\sum_{\ell<n_j}\|a_{j,\ell}\|_{L^2(I_t)}^2
\leq C_q n_j^{-1/2}\to0.
\end{equation}

For $\ell>n_j$, put $\alpha_{j,\ell}\coloneqq\sqrt{\ell^2-\lambda_j}$. 
Again $\lambda_j=n_j^2+O(1)$ gives, for all large $j$,
\[
\alpha_{j,\ell}^2\geq (n_j+1)^2-n_j^2-C
=2n_j+1-C,
\]
and hence $\alpha_{j,\ell}\geq c_0\sqrt{n_j}$.  Choose an interval $J\Subset I_t$ with $\supp q\subseteq J$.  By \eqref{eq:coeff-ode}, the coefficient satisfies $a_{j,\ell}''=\alpha_{j,\ell}^2 a_{j,\ell}$.  Therefore \Cref{lem:evanescent} gives, after squaring the estimate and renaming constants,
\[
\|a_{j,\ell}\|_{L^2(J)}^2
\leq C e^{-c\alpha_{j,\ell}}\|a_{j,\ell}\|_{L^2(I_t)}^2
\leq C e^{-c'\sqrt{n_j}}\|a_{j,\ell}\|_{L^2(I_t)}^2.
\]
Consequently,
\begin{align}\label{eq:higher-modes}
\left|\sum_{\ell>n_j}\int_{I_t}q(x)\,|a_{j,\ell}(x)|^2\dd x\right|
&\leq \|q\|_\infty \sum_{\ell>n_j}\|a_{j,\ell}\|_{L^2(J)}^2 \\
&\leq C e^{-c'\sqrt{n_j}}
\sum_{\ell>n_j}\|a_{j,\ell}\|_{L^2(I_t)}^2
\leq C e^{-c'\sqrt{n_j}}\to0.\nonumber
\end{align}
Combining \eqref{eq:Q-decomposition}, \eqref{eq:target-limit}, \eqref{eq:lower-modes}, and \eqref{eq:higher-modes} yields $\langle Q u_j,u_j\rangle_{L^2(S_t)}\to \gamma\neq0$. 
Together with \eqref{eq:mean-zero-Q}, this proves the proposition.
\end{proof}

\begin{remark}\label{rem:fixed-Q}
The observable $Q$ is chosen after passing to the subsequence on which $b_j\to b_\infty$.  Once chosen, it is fixed for all members of that subsequence.  This is exactly what is required to contradict physical equidistribution along that subsequence.
\end{remark}

\section{From a smooth observable to the main theorem}

The preceding sections produce, in each of Hassell's two alternatives, a
smooth physical observable for which the expectations fail to converge to the
Lebesgue average.  To obtain the regional form of physical scarring in
\Cref{thm:main}, we now pass from such an observable to a set.  The
mechanism is elementary: the layer-cake formula expresses the integral of a
bounded real-valued function in terms of the distribution functions of its
superlevel sets, so failure of convergence of the integral must be detected
at some level.  Choosing a regular level ensures that the resulting region
has smooth relative boundary in the interior of the billiard domain.

The following lemma records this reduction in a form convenient for the
eigenfunction application.

For a smooth function $F$ on $\Omega^\circ$, a real number $s$ is called a regular value of $F$ if $dF_z\neq0$ for every $z\in\Omega^\circ$ with $F(z)=s$; equivalently in the present planar setting, $\nabla F(z)\neq0$ at every point of the level set $F^{-1}(s)$.  If $F^{-1}(s)$ is empty, this condition is vacuous.

\begin{lemma}\label{lem:level-set}
Let $\Omega\subseteq\mathbb R^2$ be a bounded Lebesgue-measurable set with $0<|\Omega|<\infty$ and $|\partial\Omega|=0$, let $\nu_j$ be probability measures on $\Omega$, and let $m_\Omega\coloneqq|\Omega|^{-1}\mathbf 1_\Omega\,\dd z$.  Let $Q$ be a bounded real-valued Borel function on $\Omega$ whose restriction to $\Omega^\circ$ is smooth.  If $\int_\Omega Q\dd\nu_j\not\to \int_\Omega Q\dd m_\Omega$, then there is a regular value $s$ of $Q|_{\Omega^\circ}$ such that, for $A_s\coloneqq\{z\in\Omega: Q(z)>s\}$, one has $\nu_j(A_s)\not\to m_\Omega(A_s)$. 
The relative boundary of $A_s\cap\Omega^\circ$ in $\Omega^\circ$ is $\partial_{\Omega^\circ}(A_s\cap\Omega^\circ) =\{Q=s\}\cap\Omega^\circ$, 
and is smooth.  If, in addition, all $\nu_j$ are absolutely continuous with respect to Lebesgue measure, as in the eigenfunction application below, then $A_s$ may be modified on Lebesgue-null subsets of the ambient plane, in particular on boundary pieces, without changing any of the masses.
\end{lemma}

\begin{proof}
Choose $M$ so that $|Q(z)|<M$ for every $z\in\Omega$.  For every probability measure $\nu$ on $\Omega$,
\begin{equation}\label{eq:layer-cake-real}
\int_\Omega Q\dd\nu
=-M+\int_{-M}^{M}\nu(\{Q>s\})\dd s.
\end{equation}
This is the standard layer-cake identity for a bounded real-valued random variable.

Suppose, to the contrary, that $\nu_j(\{Q>s\})\to m_\Omega(\{Q>s\})$ for every regular value $s$ of $Q|_{\Omega^\circ}$.  By Sard's theorem \cite{Sard1942}, the set of critical values has Lebesgue measure zero.  Hence the pointwise convergence above holds for almost every $s\in(-M,M)$.  Since all terms are bounded by $1$, dominated convergence in \eqref{eq:layer-cake-real} gives $\int_\Omega Q\dd\nu_j\to \int_\Omega Q\dd m_\Omega$, a contradiction.  Thus some regular value $s$ gives non-convergence of the superlevel-set masses.  The smoothness of the interior level set follows from the regular-value theorem \cite[Ch.~5]{Lee2013}.
\end{proof}

\begin{proof}[Proof of \Cref{thm:main}]
Let $Z_2^{\mathrm{good}}\subseteq Z_2$ be the full-measure subset on which \Cref{prop:hassell-input}(b) holds, and set $\mathcal G\coloneqq Z_1\cup Z_2^{\mathrm{good}}$.
Since $Z_1\cup Z_2=[1,2]$, the set $\mathcal G$ has full measure in $[1,2]$.  Fix $t\in\mathcal G$.
Set $m_{S_t}\coloneqq |S_t|^{-1}\dd z$.

First suppose $t\in Z_1$.  By \Cref{prop:hassell-input}(a), there is a normalized eigenfunction sequence $w_j$ such that
$\langle Q_0 w_j,w_j\rangle_{L^2(S_t)}\to0$ for every $Q_0\in C_c^\infty(S_t^\circ\setminus\overline{R_t})$.  Choose a non-negative function $Q_0$, not identically zero, supported compactly in the interiors of the two caps, and extend $Q_0$ by zero to $S_t$.  Since $Q_0$ is supported away from the cap interfaces and the billiard boundary, this extension is smooth on $\overline S_t$.  Thus $\langle Q_0 w_j,w_j\rangle\to0$.
To obtain real eigenfunctions, write $w_j=p_j+i s_j$ with $p_j,s_j$ real.  Since the Dirichlet Laplacian is real, both $p_j$ and $s_j$ are eigenfunctions.  Choose $h_j\in\{p_j,s_j\}$ with $\|h_j\|_2\geq 1/\sqrt2$ and set $u_j\coloneqq h_j/\|h_j\|_2$.  Then $u_j$ is a normalized real eigenfunction and, since $Q_0\geq0$,
\[
\langle Q_0 u_j,u_j\rangle
\leq 2\langle Q_0 w_j,w_j\rangle
\to0.
\]
On the other hand, $\frac1{|S_t|}\int_{S_t} Q_0(z)\dd z>0$.  Thus physical equidistribution fails.  To obtain the mean-zero observable asserted in the theorem, set $Q\coloneqq Q_0-\frac1{|S_t|}\int_{S_t} Q_0(z)\dd z$.  Then $\int_{S_t} Q=0$ and $\langle Q u_j,u_j\rangle\to-\frac1{|S_t|}\int_{S_t} Q_0(z)\dd z\neq0$.
Since $\int_{S_t} Q\,\dd m_{S_t}=0$, applying \Cref{lem:level-set} to the probability measures $|u_j|^2\dd z$ and to $m_{S_t}$ gives a physical region $A\subseteq S_t$ with smooth relative boundary in the interior of the stadium such that
\[
\int_A |u_j|^2\dd z\not\to\frac{|A|}{|S_t|}.
\]
This proves the theorem in the $Z_1$ case.

Now suppose $t\in Z_2^{\mathrm{good}}$.  Then there are eigenfunctions $u_j$ and integers $n_j\to\infty$ satisfying the hypotheses of \Cref{prop:fixed-overlap}.  Applying that proposition, after passing to a subsequence we obtain a smooth real-valued mean-zero observable $Q$ with $\langle Q u_j,u_j\rangle\to \gamma\neq0$.
Since $\int_{S_t} Q\dd z=0$, this contradicts physical equidistribution.  Applying \Cref{lem:level-set} to the probability measures $|u_j|^2\dd z$ and to $m_{S_t}$ gives a physical region $A\subseteq S_t$ with smooth relative boundary in the interior of the stadium such that
\[
\int_A |u_j|^2\dd z\not\to \frac{|A|}{|S_t|}.
\]
This proves the theorem in the $Z_2$ case as well.
\end{proof}

\begin{remark}\label{rem:not-full-rectangle-localization}
The argument proves the physical-space statement but not the stronger assertion $\int_{R_t}|u_j|^2\to1$.
In the bounded-window case, localization can be distributed across a finite cluster in phase space.  \Cref{prop:fixed-overlap} avoids this by using the exact vertical Fourier coefficient selected by the fixed bouncing-ball overlap.
\end{remark}

\section*{Acknowledgements}

H.L. was supported by the National Natural Science Foundation of China (12501487), by the China Scholarship Council and IBS-R029-C4.
X.L. was supported by the Excellent Young Talents Program (Overseas) of the National Natural Science Foundation of China.

\section*{Declaration on the use of AI}

The authors used generative AI tools to assist in discussing proof strategies, checking proofs, and improving exposition.

\bibliographystyle{abbrv}
\bibliography{stadium}

\end{document}